\documentclass[draft]{amsart}

\usepackage{amssymb}
\usepackage{url} 
\usepackage{graphicx}
\usepackage[all]{xy}

\usepackage{mathrsfs}

\newcommand{\A}{\mathscr{A}}
\newcommand{\C}{\mathscr{C}}
\newcommand{\add}{\mathrm{add}}
\newcommand{\Ext}{\mathrm{Ext}}
\newcommand{\Hom}{\mathrm{Hom}}
\newcommand{\Ob}{\mathrm{Ob}}
\newcommand{\Twin}{(\Scal,\Tcal),(\Ucal,\Vcal)}

\newcommand{\Hcal}{\mathcal{H}}
\newcommand{\Kcal}{\mathcal{K}}
\newcommand{\Mcal}{\mathcal{M}}
\newcommand{\Ncal}{\mathcal{N}}
\newcommand{\Scal}{\mathcal{S}}
\newcommand{\Tcal}{\mathcal{T}}
\newcommand{\Ucal}{\mathcal{U}}
\newcommand{\Vcal}{\mathcal{V}}
\newcommand{\Wcal}{\mathcal{W}}
\newcommand{\Xcal}{\mathcal{X}}
\newcommand{\Ycal}{\mathcal{Y}}
\newcommand{\Zcal}{\mathcal{Z}}

\newtheorem{thm}{Theorem}[section]
\newtheorem{cor}[thm]{Corollary}
\newtheorem{prop}[thm]{Proposition}
\newtheorem{lem}[thm]{Lemma}
\newtheorem{claim}[thm]{Claim}

\theoremstyle{remark}
\newtheorem{rem}[thm]{Remark}

\theoremstyle{definition}
\newtheorem{dfn}[thm]{Definition}
\newtheorem{ex}[thm]{Example}

\newtheorem{fact}[thm]{Fact}

\numberwithin{equation}{section}
\numberwithin{thm}{section}

\begin{document}
%\subjclass[2000]{,}

\title[General heart construction for twin torsion pairs]{General heart construction for twin torsion pairs on triangulated categories}

\author{Hiroyuki NAKAOKA}
\address{Department of Mathematics and Computer Science, Kagoshima University, 1-21-35 Korimoto, Kagoshima, 890-0065 Japan}
\email{nakaoka@sci.kagoshima-u.ac.jp}

%\thanks{The author wishes to thank Professor Toshiyuki Katsura for his encouragement}

\thanks{Supported by JSPS Grant-in-Aid for Young Scientists (B) 22740005}

\begin{abstract}
In our previous article, we constructed an abelian category from any torsion pair on a triangulated category. This generalizes the heart of a $t$-structure and the ideal quotient by a cluster tilting subcategory.
Recently, generalizing the quotient by a cluster tilting subcategory, Buan and Marsh showed that an integral preabelian category can be constructed as a quotient, from a rigid object in a triangulated category with some conditions.
In this article, by considering a pair of torsion pairs, we make a simultaneous genralization of these two constructions.
\end{abstract}

\maketitle

\section{Introduction and preliminaries}
For any category $\Kcal$, we write abbreviately $K\in\Kcal$, to indicate that $K$ is an object of $\Kcal$.
For any $K,L\in\Kcal$, let $\Kcal(K,L)$ denote the set of morphisms from $K$ to $L$.
If $\Mcal,\Ncal$ are full subcategories of $\Kcal$, then $\Kcal(\Mcal,\Ncal)=0$ means that $\Kcal(M,N)=0$ for any $M\in\Mcal$ and $N\in\Ncal$. For each $K\in\Ob(\Kcal)$, similarly $\Kcal(\Mcal,K)=0$ means that $\Kcal(M,K)=0$ for any $M\in\Mcal$. We denote the full subcategory of $\Kcal$ consisting of those $K\in\Kcal$ satisfying $\Kcal(\Mcal,K)=0$ by $\Mcal^{\perp}$. Dually, $\Kcal(K,\Ncal)=0$ means $\Kcal(K,N)=0$ for any $N\in\Ncal$, and these form a full subcategory ${}^{\perp}\!\Ncal\subseteq\Kcal$.
If $\Kcal$ is additive and $\Ncal\subseteq\Kcal$ is a full additive subcategory, then $\Kcal/\Ncal$ is defined to be the ideal quotient of $\Kcal$ by $\Ncal$. Namely, $\Kcal/\Ncal$ is an additive category defined by\\ 
\ -\ \  $\Ob(\Kcal/\Ncal)=\Ob(\Kcal)$,\\
\ -\ \  For any $X,Y\in\Kcal$,
\[ (\Kcal/\Ncal)(X,Y)=\Kcal(X,Y)/\{ f\in\Kcal(X,Y) \mid f\ \text{factors through some}\ N\in\Ncal\}. \]

Throughout this article, we fix a triangulated category $\C$. Any subcategory of $\C$ is a full, additive subcategory closed under isomorphisms and direct summands. For any object $T\in\C$, $\add (T)$ denotes the full subcategory of $\C$ consisting of direct summands of finite direct sums of $T$.
When $\Mcal,\Ncal$ are subcategories of $\C$ and $C\in\C$, then the abbreviations $\Ext^1(\Mcal,\Ncal)=0$ and $\Ext^1(\Mcal,C)=0$ and $\Ext^1(C,\Ncal)=0$ are defined similarly as above.
For any pair of subcategories $\Mcal,\Ncal\subseteq\C$, we define $\Mcal\ast\Ncal\subseteq\C$ as the full subcategory consisting of those $C\in\C$ admitting some distinguished triangle
\[ M\rightarrow C\rightarrow N\rightarrow M[1] \]
with $M\in\Mcal$ and $N\in\Ncal$.

By definition \cite{IY}, a torsion pair $(\Xcal,\Ycal)$ on $\C$ is a pair of (full additive) subcategories $\Xcal,\Ycal\subseteq\C$ satisfying
\begin{itemize}
\item[{\rm (i)}] $\C(\Xcal,\Ycal)=0$,
\item[{\rm (ii)}] $\C=\Xcal\ast\Ycal$.
\end{itemize}
Previously in \cite{N}, we showed that if we are given a torsion pair $(\Xcal,\Ycal)$ on $\C$, then
\begin{equation}\label{GeneralHeart}
(\, (\Xcal\ast\Wcal)\cap (\Wcal\ast\Ycal[1])\, )/\Wcal
\end{equation}
becomes an abelian category, where $\Wcal=(\Xcal[1]\cap\Ycal)$.
This generalizes the following two constructions.
\begin{enumerate}
\item {\bf The heart of a $t$-structure}. A $t$-structure is nothing other than a torsion pair $(\Xcal,\Ycal)$ satisfying $\Xcal[1]\subseteq\Xcal$. In this case, $(\ref{GeneralHeart})$ agrees with the heart \cite{BBD}. 
\item {\bf Ideal quotient by a cluster tilting subcategory}. A full additive subcategory $\Tcal\subseteq\C$ is a cluster tilting subcategory if and only if $(\Tcal[-1],\Tcal)$ is a torsion pair on $\C$. In this case, $(\ref{GeneralHeart})$ agrees with the ideal quotient of $\C$ by $\Tcal$, which was shown to become abelian in \cite{KZ} (originally in \cite{BMR}, or \cite{KR} in $2$-$\mathrm{CY}$ case). 
\end{enumerate}

Recently, generalizing the second case of a cluster tilting subcategory, Buan and Marsh showed that if $T$ is a rigid (i.e. $\Ext^1(T,T)=0$) object in a $\Hom$-finite Krull-Schmidt triangulated category $\C$  (over a field $k$) with a Serre functor, then
\begin{equation}\label{BMcat}
\qquad\qquad\qquad\C/\Xcal_T\qquad (\text{where}\ \Xcal_T=(\add (T))^{\perp})
\end{equation}
becomes an integral preabelian category. %Here, $\Xcal_T$ denotes the full subcategory of $\C$ consisting of those $C\in\C$ satisfying $\Ext^1(T,C)=0$.
(In the notation of \cite{BM}, $\Xcal_T$ is written as $\Xcal_T=(\add (T))^{\perp}[1]$. This is only due to the difference in the definition of $\Mcal^{\perp}$. In \cite{BM}, for any $\Mcal\subseteq\C$, $\Mcal^{\perp}$ is defined to be the full subcategory of $\C$ consisting of those $C\in\C$ satisfying $\Ext^1(\Mcal,C)$=0. )

As in \cite{R} (and as quoted in \cite{BM}), an additive category $\A$ is {\it preabelian} if any morphism in $\A$ has a kernel and a cokernel. A preabelian category is {\it left semi-abelian} if and only if for any pullback diagram
\begin{equation}
\label{PullBackDiag}
\xy
(-6,6)*+{A}="0";
(6,6)*+{B}="2";
(-6,-6)*+{C}="4";
(6,-6)*+{D}="6";
%(9,-7)*+{,}="7";
(0,0)*+{\square}="16";
{\ar^{\alpha} "0";"2"};
{\ar_{\beta} "0";"4"};
{\ar^{\gamma} "2";"6"};
{\ar_{\delta} "4";"6"};
\endxy
\end{equation}
in $\A$, \lq\lq $\delta$ is a cokernel morphism" implies \lq\lq $\alpha$ is epimorphic" \cite{R}. A {\it right semi-abelian} category is characterized dually, using pushout diagrams. A {\it semi-abelian} category is defined to be a preabelian category which is both left semi-abelian and right semi-abelian.

A preabelian category $\A$ is {\it left integral} if for any pullback diagram $(\ref{PullBackDiag})$, \lq\lq $\delta$ is epimorphic" implies \lq\lq $\alpha$ is epimorphic". A {\it right integral} category is defined dually, using pushout diagrams. An {\it integral} category is defined to to be a preabelian category which is both left integral and right integral. Thus a preabelian category $\A$ is semi-abelian whenever it is integral.
Moreover if $\A$ is integral, then the localization of $\A$ by regular ($=$ epimorphic and monomorphic) morphisms are shown to become abelian. In \cite{BM}, using this fact, Buan and Marsh made an abelian category out of their integral preabelian category $\C/\Xcal_T$.

In this article, we generalize simultaneously Buan and Marsh's $\C/\Xcal_T$ and our $\underline{\Hcal}$, using a pair of torsion pairs. Starting from torsion pairs, we need no assumption on $\C$, except it is triangulated.

\section{Definition and basic properties}

As before, $\C$ is a fixed triangulated category. Any subcategory of $\C$ is assumed to be full, additive, closed under isomorphisms and direct summands.

\begin{dfn}\label{DefDivide}
Let $\Ucal$ and $\Vcal$ be full additive subcategories of $\C$.
We call $(\Ucal,\Vcal)$ a {\it cotorsion pair} if it satisfies
\begin{itemize}
\item[{\rm (i)}] $\Ext^1(\Ucal,\Vcal)=0$,
\item[{\rm (ii)}] $\C=\Ucal\ast\Vcal[1]$.
\end{itemize}
\end{dfn}

\begin{rem} $(\Ucal,\Vcal)$ is a cotorsion pair if and only if $(\Ucal[-1],\Vcal)$ is a {\it torsion pair} in \cite{IY}. (Unlike \cite{BR}, it does not require the closedness under shifts.)
In this sense, a cotorsion pair is essentially the same as a torsion pair. However we prefer cotorsion pairs, for the sake of duality in shifts.
\end{rem}

\begin{rem} For any cotorsion pair $(\Ucal,\Vcal)$ on $\C$, we have $\Ucal={}^{\perp}\!(\Vcal[1])$ and $\Vcal=(\Ucal[-1])^{\perp}$.
%\[ \Ucal={}^{\perp}\!(\Vcal[1]),\qquad \Vcal=(\Ucal[-1])^{\perp}. \]
\end{rem}

Cotorsion pairs generalize $t$-structures and cluster tilting subcategories, as follows.
\begin{ex}(cf. Definition 2.6 in \cite{ZZ}, Proposition 2.6 in \cite{N}) \label{ExCotors}
$\ \ $
\begin{enumerate}
\item A $t$-{\it structure} is a pair of subcategories $(\Xcal,\Ycal)$ where $(\Ucal,\Vcal)=(\Xcal[1],\Ycal)$ is a cotorsion pair satisfying $\Ucal[1]\subseteq\Ucal$. This is also equivalent to $\Vcal[-1]\subseteq\Vcal$.
\item A {\it co}-$t$-{\it structure} is a pair of subcategories $(\Xcal,\Ycal)$ where $(\Ucal,\Vcal)=(\Xcal[1],\Ycal)$ is a cotorsion pair satisfying $\Ucal[-1]\subseteq\Ucal$. This is also equivalent to $\Vcal[1]\subseteq\Vcal$.
\item A cotorsion pair $(\Ucal,\Vcal)$ is called {\it rigid} if $\Ext^1(\Ucal,\Ucal)=0$. This is also equivalent to $\Ucal\subseteq\Vcal$.
\item A subcategory $\Tcal\subseteq\C$ is a {\it cluster tilting} subcategory if and only if $(\Tcal,\Tcal)$ is a cotorsion pair.
\end{enumerate}
\end{ex}
\begin{rem}
Using a result in \cite{AN}, we can characterize a co-$t$-structure by the vanishing of an abelian category $\underline{\Hcal}$ defined as below. In fact, a cotorsion pair $(\Ucal,\Vcal)$ becomes a co-$t$-structure if and only if it satisfies $\underline{\Hcal}=0$.
\end{rem}

In \cite{N}, we showed the following.
\begin{thm}\label{ThmGHC1}$($Theorem 6.4 in \cite{N}$)$
Let $(\Ucal, \Vcal)$ be a cotorsion pair on $\C$. If we define full subcategories of $\C$ by
\[ \Wcal=\Ucal\cap\Vcal,\ \ \C^-=\Ucal[-1]\ast\Wcal,\ \ \C^+=\Wcal\ast\Vcal[1],\ \ \Hcal=\C^+\cap\C^-, \]
then the ideal quotient $\Hcal/\Wcal$ becomes abelian.
\end{thm}

In this article, we generalize this to the case of pairs of cotorsion pairs.
We work on a pair of cotorsion pairs $(\Scal,\Tcal),(\Ucal,\Vcal)$ on $\C$ satisfying $\Ext^1(\Scal,\Vcal)=0$.
Since a \lq\lq pair of pairs" is a bit confusing, we use the following terminology.

\begin{dfn}\label{DefTwin}
A pair of cotorsion pairs $(\Scal,\Tcal),(\Ucal,\Vcal)$ on $\C$ is called a {\it twin cotorsion pair} if it satisfies
\begin{equation}\label{Twin}
\Ext^1(\Scal,\Vcal)=0.
\end{equation}
Remark that this condition is equal to $\Scal\subseteq\Ucal$, and also to $\Tcal\supseteq\Vcal$.
\end{dfn}

\begin{dfn}
Let $(\Scal,\Tcal),(\Ucal,\Vcal)$ be a twin cotorsion pair on $\C$. We define subcategories of $\C$ by
\[ \Wcal=\Tcal\cap\Ucal,\ \ \C^-=\Scal[-1]\ast\Wcal,\ \  \C^+=\Wcal\ast\Vcal[1],\ \ \Hcal=\C^+\cap\C^-. \]
Each of $\C,\C^+,\C^-,\Hcal$ contains $\Wcal$ as a subcategory. We denote their ideal quotients by $\Wcal$ by $\underline{\C}, \underline{\C}^+,\underline{\C}^-,\underline{\Hcal}$. Thus we obtain a sequence of full additive subcategories
\[ \underline{\Hcal}\subseteq\underline{\C}^+\subseteq\underline{\C},\quad \underline{\Hcal}\subseteq\underline{\C}^-\subseteq\underline{\C}. \]
For any morphism $f\in\C(A,B)$, we denote its image in $\underline{\C}(A,B)$ by $\underline{f}$.
\end{dfn}

\begin{rem}\label{TrivW}
Since $\Wcal$ is closed under direct summands, for any $C\in\C$ we have
\[ C= 0\ \ \text{in}\ \  \underline{\C}\ \ \Longleftrightarrow\ \ C\in\Wcal. \]
\end{rem}

\begin{ex}\label{ExTwin}
$\ \ $
\begin{enumerate}
\item A single cotorsion pair can be regarded as a degenerated case of a twin cotorsion pair. A twin cotorsion pair $\Twin$ is a single cotorsion pair (namely, $(\Scal,\Tcal)=(\Ucal,\Vcal)$) if and only if $\Scal=\Ucal$ if and only if $\Tcal=\Vcal$. In this case, $($since $\Wcal,\C^+,\C^-$ and $\Hcal$ agrees with those in Theorem \ref{ThmGHC1},$)$ $\underline{\Hcal}$ becomes abelian as in Theorem \ref{ThmGHC1}. 
\item Another extremal case is when $\Tcal=\Ucal$. Remark that in this case, $\Scal\subseteq\Tcal$ and $\Ucal\supseteq\Vcal$ holds. In particular, $(\Scal,\Tcal)$ is rigid.

As shown in \cite{BM}, if $T$ is a rigid object in a $\Hom$-finite Krull-Schmidt triangulated category $\C$ (over a field $k$) with a Serre functor, then $(\Scal,\Tcal)=(\add(T)[1],\Xcal_T)$ and $(\Ucal,\Vcal)=(\Xcal_T,(\Xcal_T)^{\perp}[-1])$ are cotorsion pairs (Lemma 1.2 in \cite{BM}). Since $T$ is rigid, this pair satisfies $\add (T)[1]\subseteq \Xcal_T$. Thus $\Ext^1(\add (T)[1],(\Xcal_T)^{\perp}[-1])=0$, and $(\add(T)[1],\Xcal_T),(\Xcal_T,(\Xcal_T)^{\perp}[-1])$ becomes a twin cotorsion pair.

In this case, we have $\underline{\Hcal}=\C/\Xcal_T$, and it was shown in \cite{BM} that this category becomes integral preabelian.
(Remark that when $\Tcal=\Ucal$, generally we have $\Wcal=\Tcal=\Ucal$ and $\C^+=\C^-=\C$.)
\end{enumerate}
\end{ex}

\begin{rem}\label{RemTTF}
A similar situation to {\rm (2)} in Example \ref{ExTwin} appears in \cite{BR} as a $\mathrm{TTF}$-triple. A $\mathit{TTF}$-{\it triple} on $\C$ is a triplet $(\Xcal,\Ycal,\Zcal)$ of subcategories of $\C$, in which both $(\Xcal,\Ycal)$ and $(\Ycal,\Zcal)$ are $t$-structures.

\begin{lem}\label{LemAB}
$\ \ $
\begin{enumerate}
\item If $\ U[-1]\rightarrow A\overset{f}{\longrightarrow}B\rightarrow U$ is a distinguished triangle in $\C$ satisfying $U\in\Ucal$, then $A\in\C^-$ implies $B\in\C^-$.
\item If $\ S[-1]\rightarrow A\overset{f}{\longrightarrow}B\rightarrow S$ is a distinguished triangle in $\C$ satisfying $S\in\Scal$, then $B\in\C^-$ implies $A\in\C^-$.
\end{enumerate}
\end{lem}
\begin{proof}
{\rm (1)} Take distinguished triangles
\begin{eqnarray*}
S_A[-1]\overset{s_A}{\longrightarrow}A\overset{w_A}{\longrightarrow}W_A\rightarrow S_A&\ &(S_A\in\Scal, W_A\in\Wcal),\\
S_B[-1]\overset{s_B}{\longrightarrow}B\overset{t_B}{\longrightarrow}T_B\rightarrow S_B&\ &(S_B\in\Scal, T_B\in\Tcal).
\end{eqnarray*}
Since $\Ext^1(S_A,T_B)=0$, $f$ induces a morphism of triangles
\[
\xy
(-22,6)*+{S_A[-1]}="0";
(-7,6)*+{A}="2";
(6,6)*+{W_A}="4";
(18,6)*+{S_A}="6";
(-22,-6)*+{S_B[-1]}="10";
(-7,-6)*+{B}="12";
(6,-6)*+{T_B}="14";
(18,-6)*+{S_B.}="16";
{\ar^{s_A} "0";"2"};
{\ar^{w_A} "2";"4"};
{\ar^{} "4";"6"};
{\ar^{} "0";"10"};
{\ar^{f} "2";"12"};
{\ar^{} "4";"14"};
{\ar^{} "6";"16"};
{\ar_{s_B} "10";"12"};
{\ar_{t_B} "12";"14"};
{\ar_{} "14";"16"};
{\ar@{}|\circlearrowright "0";"12"};
{\ar@{}|\circlearrowright "2";"14"};
{\ar@{}|\circlearrowright "4";"16"};
\endxy
\]
It suffices to show $T_B\in\Ucal$.

Take any $V^{\dag}\in\Vcal$, and any $v\in\C(T_B,V^{\dag}[1])$. Since $\Ext^1(W_A,V^{\dag})=0$, we have $v\circ t_B\circ f=0$.
Applying this to the given distinguished triangle, we see that $v\circ t_B$ factors through $U$,
\[
\xy
(-18,6)*+{U[-1]}="0";
(-6,6)*+{A}="2";
(6,6)*+{B}="4";
(14,-8)*+{}="5";
(18,6)*+{U}="6";
(6,-6)*+{T_B}="8";
(6,-18)*+{V^{\dag}[1]}="10";
{\ar^{} "0";"2"};
{\ar^{f} "2";"4"};
{\ar^{} "4";"6"};
{\ar_{t_B} "4";"8"};
{\ar_{} "6";"10"};
{\ar_{v} "8";"10"};
{\ar@{}|\circlearrowright "4";"5"};
\endxy
\]
and $v\circ t_B=0$ follows from $\Ext^1(U,V^{\dag})=0$.
Thus $v$ factors through $S_B$,
\[
\xy
(-23,6)*+{S_B[-1]}="0";
(-7,6)*+{B}="2";
(6,6)*+{T_B}="4";
(15,-3)*+{}="5";
(19,6)*+{S_B}="6";
(6,-8)*+{V^{\dag}[1]}="8";
{\ar^{} "0";"2"};
{\ar^{t_B} "2";"4"};
{\ar^{} "4";"6"};
{\ar_{v} "4";"8"};
{\ar_{} "6";"8"};
{\ar@{}|\circlearrowright "4";"5"};
\endxy
\]
which means $v=0$, since $\Ext^1(S_B,V^{\dag})=0$.

{\rm (2)}
Take distinguished triangles
\begin{eqnarray*}
&S_B[-1]\overset{s_B}{\longrightarrow}B\overset{w_B}{\longrightarrow}W_B\rightarrow S_B\quad(S_B\in\Scal, W_B\in\Wcal),&\\
&X\rightarrow A\overset{w_B\circ f}{\longrightarrow}W_B\rightarrow X[1].&
\end{eqnarray*}
Then by the octahedral axiom,
\[ S[-1]\rightarrow X\rightarrow S_B[-1]\rightarrow S \]
is also a distinguished triangle. This implies $X\in\Scal[-1]$, and thus $A\in\C^-$.
\[
\xy
(-20,16)*+{S[-1]}="0";
(0.5,15)*+{X}="2";
(17,14)*+{S_B[-1]}="4";
(-2,4)*+{A}="6";
(6.1,-0.6)*+{B}="8";
(-5.8,-14.2)*+{W_B}="10";
(7.5,8.5)*+_{_{\circlearrowright}}="12";
(-5.5,11.5)*+_{_{\circlearrowright}}="14";
(-0.5,-4.3)*+_{_{\circlearrowright}}="14";
{\ar^{} "0";"2"};
{\ar^{} "2";"4"};
{\ar_{} "0";"6"};
{\ar^{} "2";"6"};
{\ar^{f} "6";"8"};
{\ar_{} "6";"10"};
{\ar^{} "4";"8"};
{\ar^{w_B} "8";"10"};
\endxy
\]
\end{proof}

Dually, the following holds.
\begin{lem}\label{LemABDual}
$\ \ $
\begin{enumerate}
\item If $\ T\rightarrow A\overset{f}{\longrightarrow}B\rightarrow T[1]$ is a distinguished triangle in $\C$ satisfying $T\in\Tcal$, then $B\in\C^+$ implies $A\in\C^+$.
\item If $\ V\rightarrow A\overset{f}{\longrightarrow}B\rightarrow V[1]$ is a distinguished triangle in $\C$ satisfying $V\in\Vcal$, then $A\in\C^+$ implies $B\in\C^+$.
\end{enumerate}
\end{lem}

The following Lemma is trivial.
\begin{lem}\label{LemFactor}
Let $S_X[-1]\overset{s_X}{\longrightarrow}X\overset{t_X}{\longrightarrow}T_X\rightarrow S_X$ be a distinguished triangle, with $S_X\in\Scal$ and $T_X\in\Tcal$.
If a morphism $x\in\C(X,Y)$ factors through some $T\in \Tcal$, then $x$ factors through $T_X$.

Similar statement also holds for a distinguished triangle
\[ U_X[-1]\overset{u_X}{\longrightarrow}X\overset{v_X}{\longrightarrow}U_X\rightarrow V_X\quad (U_X\in\Ucal,V_X\in\Vcal) \]
and a morphism $x\in\C(X,Y)$ factoring some $V\in\Vcal$.
\end{lem}
\begin{proof}
If $x$ factors through $T\in\Tcal$, then $x\circ s_X=0$ follows from $\Ext^1(S_X,T)=0$. Thus it factors through $t_X$.
Similarly for the latter part.
\end{proof}
\begin{rem}
The dual of Lemma \ref{LemFactor} also holds.
\end{rem}

\end{rem}

\section{Adjoints}

In the following, we fix a twin cotorsion pair $(\Scal,\Tcal),(\Ucal,\Vcal)$ on $\C$. Since this assumption is self-dual, we often omit proofs of dual statements.

\begin{dfn}\label{DefKC}
For any $C\in\C$, define $K_C\in\C$ and $k_C\in\C(K_C,C)$ as follows:
\begin{itemize}
\item[1.] Take a distinguished triangle
\[ S[-1]\rightarrow C\overset{a}{\longrightarrow}T\rightarrow S\quad (S\in\Scal,T\in\Tcal) \]
\item[2.] then, take a distinguished triangle
\[ U\rightarrow T\overset{b}{\longrightarrow}V[1]\rightarrow U[1]\quad (U\in\Ucal,V\in\Vcal) \]
\item[3.] and then, take a distinguished triangle
\[ V\rightarrow K_C\overset{k_C}{\longrightarrow}C\overset{b\circ a}{\longrightarrow} V[1]. \]
\end{itemize}
By the octahedral axiom, $S[-1]\rightarrow K_C\overset{k_C}{\longrightarrow}C\rightarrow S$ 
%\[ S[-1]\rightarrow K_C\overset{k_C}{\longrightarrow}C\rightarrow S \]
is also a distinguished triangle.
\[
\xy
(-20,16)*+{S[-1]}="0";
(0.5,15)*+{K_C}="2";
(17,14)*+{U}="4";
(-2,4)*+{C}="6";
(6.1,-0.6)*+{T}="8";
(-5.8,-14.2)*+{V[1]}="10";
(7.5,8.5)*+_{_{\circlearrowright}}="12";
(-5.5,11.5)*+_{_{\circlearrowright}}="14";
(-0.5,-4.3)*+_{_{\circlearrowright}}="14";
{\ar^{} "0";"2"};
{\ar^{} "2";"4"};
{\ar_{} "0";"6"};
{\ar^{k_C} "2";"6"};
{\ar^{a} "6";"8"};
{\ar_{} "6";"10"};
{\ar^{} "4";"8"};
{\ar^{b} "8";"10"};
\endxy
\]
\end{dfn}

\begin{claim}\label{ClaimKC}
$\ \ $
\begin{enumerate}
\item For any $C\in\C$, we have $K_C\in\C^-$.
\item If $C\in\C^+$, then $K_C\in\Hcal$.
\end{enumerate}
\end{claim}
\begin{proof}
We use the notation in Definition \ref{DefKC}.

{\rm (1)} Remark that we have a distinguished triangle
\[ V\rightarrow U\rightarrow T\overset{b}{\longrightarrow}V[1]. \]
Since $\Tcal\supseteq\Vcal\ni V$ and $\Tcal$ is closed under extensions, it follows $U\in\Ucal\cap\Tcal=\Wcal$. Since $S[-1]\rightarrow K_C\rightarrow U\rightarrow S$ 
%\[ S[-1]\rightarrow K_C\rightarrow U\rightarrow S \]
is a distinguished triangle, we obtain $K_C\in\C^-$.

{\rm (2)} Since $V\rightarrow K_C\rightarrow C\rightarrow V[1]$
%\[ V\rightarrow K_C\rightarrow C\rightarrow V[1] \]
is a distinguished triangle, this immediately follows from Lemma \ref{LemABDual}.
\end{proof}

Dually, we define as follows.
\begin{dfn}\label{DefZC}
For any $C\in\C$, define $Z_C\in\C$ and $z_C\in\C(C,Z_C)$ as follows:
\begin{itemize}
\item[1.] Take a distinguished triangle
\[ V\rightarrow U\overset{a}{\longrightarrow}C\rightarrow V[1]\quad (U\in\Ucal,V\in\Vcal) \]
\item[2.] then, take a distinguished triangle
\[ T[-1]\rightarrow S[-1]\overset{b}{\longrightarrow}U\rightarrow T\quad (S\in\Scal,T\in\Tcal) \]
\item[3.] and then, take a distinguished triangle
\[ S[-1]\overset{a\circ b}{\longrightarrow}C\overset{z_C}{\longrightarrow}Z_C\rightarrow S. \]
\end{itemize}
By the octahedral axiom, $V\rightarrow T\rightarrow Z_C\rightarrow V[1]$ 
%\[ V\rightarrow T\rightarrow Z_C\rightarrow V[1] \]
is also a distinguished triangle.
\[
\xy
(-16,16)*+{S[-1]}="0";
(-15,-1)*+{U}="2";
(-13.5,-17.5)*+{T}="4";
(-4,2)*+{C}="6";
(2.3,-5.3)*+{Z_C}="8";
(16.5,8.5)*+{V[1]}="10";
(-8.5,-7.5)*+_{_{\circlearrowright}}="12";
(-11.5,5.5)*+_{_{\circlearrowright}}="14";
(4.3,0.5)*+_{_{\circlearrowright}}="14";
{\ar_{b} "0";"2"};
{\ar_{} "2";"4"};
{\ar^{} "0";"6"};
{\ar_{a} "2";"6"};
{\ar_{z_C} "6";"8"};
{\ar^{} "6";"10"};
{\ar_{} "4";"8"};
{\ar_{} "8";"10"};
\endxy
\]
\end{dfn}
\begin{claim}\label{ClaimZC}
$\ \ $
\begin{enumerate}
\item For any $C\in\C$, we have $Z_C\in\C^+$.
\item If $C\in\C^-$, then $Z_C\in\Hcal$.
\end{enumerate}
\end{claim}
\begin{proof}
This is the dual of Claim \ref{ClaimZC}.
\end{proof}

%$K_C$ satisfies the following adjointness.

\begin{prop}\label{PropKC}
For any $C\in\C$, let $K_C\overset{k_C}{\longrightarrow}C$ be as in Definition \ref{DefKC}. Then for any $X\in\C^-$,
\[ \underline{k}_C\circ-\colon\underline{\C}^-(X,K_C)\rightarrow\underline{\C}(X,C) \]
is bijective.
\end{prop}
\begin{proof}
Take a distinguished triangle
\[ S_X[-1]\overset{s_X}{\longrightarrow}X\overset{w_X}{\longrightarrow}W_X\rightarrow S_X\quad(S_X\in\Scal,W_X\in\Wcal). \]

First we show the injectivity.
Suppose $x\in\C(X,K_C)$ satisfies $\underline{k}_C\circ\underline{x}=0$. 
By definition, this means that $k_C\circ x$ factors through some object in $\Wcal$. Thus by Lemma \ref{LemFactor}, $k_C\circ x$ factors through $w_X$. This gives a morphism of triangles
\[
\xy
(-22,6)*+{S_X[-1]}="0";
(-6,6)*+{X}="2";
(6,6)*+{W_X}="4";
(18,6)*+{S_X}="6";
(-22,-6)*+{V}="10";
(-6,-6)*+{K_C}="12";
(6,-6)*+{C}="14";
(18,-6)*+{V[1].}="16";
{\ar^{s_X} "0";"2"};
{\ar^{w_X} "2";"4"};
{\ar^{} "4";"6"};
{\ar^{} "0";"10"};
{\ar^{x} "2";"12"};
{\ar^{} "4";"14"};
{\ar^{} "6";"16"};
{\ar_{} "10";"12"};
{\ar_{k_C} "12";"14"};
{\ar_{} "14";"16"};
{\ar@{}|\circlearrowright "0";"12"};
{\ar@{}|\circlearrowright "2";"14"};
{\ar@{}|\circlearrowright "4";"16"};
\endxy
\]
Since $\Ext^1(S_X,V)=0$, it follows $x\circ S_X=0$, and thus $x$ factors through $W_X$,
\[
\xy
(-26,6)*+{S_X[-1]}="0";
(-8,6)*+{X}="2";
(8,6)*+{W_X}="4";
(-4,-7)*+{K_C}="8";
(-22,10)*+{}="5";
(0,10)*+{}="7";
{\ar^{s_X} "0";"2"};
{\ar^{w_X} "2";"4"};
{\ar@/_0.60pc/_{0} "0";"8"};
{\ar_{x} "2";"8"};
{\ar_{} "4";"8"};
{\ar@{}|\circlearrowright "5";"8"};
{\ar@{}|\circlearrowright "7";"8"};
\endxy
\]
which means $\underline{x}=0$.

Second, we show the surjectivity. Take any $y\in\C(X,C)$. %Let
%\[ S_X[-1]\overset{s_X}{\longrightarrow}X\overset{w_X}{\longrightarrow}W_X\rightarrow S_X\quad (S_X\in\Scal,W_X\in\Wcal) \]
%be a distingished triangle.
In the notation of Definition \ref{DefKC},
\[
\xy
(-20,16)*+{S[-1]}="0";
(0.5,15)*+{K_C}="2";
(17,14)*+{U}="4";
(-2,4)*+{C}="6";
(5.9,-0.6)*+{T}="8";
(-5.8,-14.2)*+{V[1]}="10";
(7.5,8.5)*+_{_{\circlearrowright}}="12";
(-5.5,11.5)*+_{_{\circlearrowright}}="14";
(-0.5,-4.3)*+_{_{\circlearrowright}}="14";
{\ar^{} "0";"2"};
{\ar^{} "2";"4"};
{\ar_{} "0";"6"};
{\ar^{k_C} "2";"6"};
{\ar^{a} "6";"8"};
{\ar_{b\circ a} "6";"10"};
{\ar^{} "4";"8"};
{\ar^{b} "8";"10"};
\endxy
\]
since $\Ext^1(S_X,T)=0$, $y$ induces a morphism of triangles
\[
\xy
(-22,6)*+{S_X[-1]}="0";
(-7,6)*+{X}="2";
(6,6)*+{W_X}="4";
(18,6)*+{S_X}="6";
(-22,-6)*+{S[-1]}="10";
(-7,-6)*+{C}="12";
(6,-6)*+{T}="14";
(18,-6)*+{S.}="16";
{\ar^{s_X} "0";"2"};
{\ar^{w_X} "2";"4"};
{\ar^{} "4";"6"};
{\ar^{} "0";"10"};
{\ar^{y} "2";"12"};
{\ar^{{}^{\exists}t} "4";"14"};
{\ar^{} "6";"16"};
{\ar_{} "10";"12"};
{\ar_{a} "12";"14"};
{\ar_{} "14";"16"};
{\ar@{}|\circlearrowright "0";"12"};
{\ar@{}|\circlearrowright "2";"14"};
{\ar@{}|\circlearrowright "4";"16"};
\endxy
\]
Since $\Ext^1(W_X,V)=0$, we obtain
\[ b\circ a\circ y=b\circ t\circ w_X =0\circ w_X=0. \]
Thus $y$ factors through $k_C$.
\[
\xy
(-22,0)*+{V}="0";
(-9,0)*+{K_C}="2";
(8,0)*+{C}="4";
(12,10)*+{}="5";
(20,0)*+{V[1]}="6";
(-2,12)*+{X}="-3";
(1,-4)*+{}="3";
{\ar^{} "0";"2"};
{\ar_{k_C} "2";"4"};
{\ar_{b\circ a} "4";"6"};
{\ar^{} "-3";"2"};
{\ar@/^0.80pc/^{0} "-3";"6"};
{\ar^{y} "-3";"4"};
{\ar@{}|\circlearrowright"-3";"3"};
{\ar@{}|\circlearrowright"4";"5"};
\endxy
\]
\end{proof}

Dually, we have the following.
\begin{prop}\label{PropZC}
For any $C\in\C$, let $C\overset{z_C}{\longrightarrow} Z_C$  be as in Definition \ref{DefZC}. Then, for any $Y\in\C^+$,
\[ -\circ\underline{z}_C\colon\underline{\C}^+(Z_C,Y)\rightarrow\underline{\C}(C,Y) \]
is bijective.
\end{prop}

In the terminology of \cite{B}, Proposition \ref{PropZC} means that for any $C\in\C$, $Z_C\overset{z_C}{\longrightarrow}C$ gives a reflection of $C$ along the inclusion functor $\underline{\C}^+\hookrightarrow\underline{\C}$. As a corollary, we obtain the following.
\begin{cor}\label{CorZC}$($Proposition 3.1.2 and Proposition 3.1.3 in \cite{B}$)$
The inclusion functor $i^+\colon\underline{\C}^+\hookrightarrow\underline{\C}$ has a left adjoint $\tau^+\colon\underline{\C}\rightarrow\underline{\C}^+$. If we denote the adjunction by $\eta\colon\mathrm{Id}_{\underline{\C}}\Longrightarrow i^+\circ\tau^+$, then there exists a natural isomorphism $Z_C\cong\tau^+(C)$ in $\underline{\C}$, compatible with $\underline{z}_C$ and $\eta_C$.
\[
\xy
(-8,8)*+{C}="0";
(8,8)*+{Z_C}="2";
(0,-5)*+{\tau^+(C)}="4";
(0,12)*+{}="6";
{\ar^{\underline{z_C}} "0";"2"};
{\ar_{\eta_C} "0";"4"};
{\ar^{\cong} "2";"4"};
{\ar@{}|\circlearrowright "6";"4"};
\endxy
\]
In particular, $Z_C$ is determined up to an isomorphism in $\underline{\C}$.
\end{cor}

Dually the following holds.
\begin{cor}\label{CorKC}
The inclusion functor $i^-\colon\underline{\C}^-\hookrightarrow\underline{\C}$ has a right adjoint $\tau^-\colon\underline{\C}\rightarrow\underline{\C}^-$. For any $C\in\C$, there is a natural isomorphism $K_C\cong \tau^-(C)$ in $\underline{\C}^-$.
\end{cor}

\begin{lem}\label{LemUZ}
For any $C\in\C$, the following are equivalent.
\begin{enumerate}
\item $\tau^+(C)=0$.
\item $Z_C\in\Wcal$
\item $C\in\Ucal$.
\end{enumerate}
\end{lem}
\begin{proof}
The equivalence between {\rm (1)} and {\rm (2)} follows from Remark \ref{TrivW} and Corollary \ref{CorZC}.

Suppose {\rm (2)} holds. Then, in the notation of the following diagram,
\begin{equation}\label{RefDiag}
\xy
(-16,16)*+{S[-1]}="0";
(-15,-1)*+{U}="2";
(-13.5,-17.5)*+{T}="4";
(-4,2)*+{C}="6";
(2.3,-5.3)*+{Z_C}="8";
(16.5,8.5)*+{V[1]}="10";
(-8.5,-7.5)*+_{_{\circlearrowright}}="12";
(-11.5,5.5)*+_{_{\circlearrowright}}="14";
(4.3,0.5)*+_{_{\circlearrowright}}="14";
{\ar_{b} "0";"2"};
{\ar_{} "2";"4"};
{\ar^{} "0";"6"};
{\ar_{a} "2";"6"};
{\ar_{z_C} "6";"8"};
{\ar^{c} "6";"10"};
{\ar_{e} "4";"8"};
{\ar_{d} "8";"10"};
\endxy
\end{equation}
we have $d=0$ since $\Ext^1(Z_C,V)=0$. Thus it follows $c=0$ and $C$ becomes a direct summand of $U$, which means $C\in\Ucal$.

Conversely, suppose {\rm (3)} holds. Then in the notation in $(\ref{RefDiag})$, we may assume $a$ and $e$ are isomorphisms. Since $T\in\Wcal$, we obtain $Z_C\in\Wcal$.
\end{proof}

Dually we have the following.
\begin{lem}\label{LemTK}
For any $C\in\C$, the following are equivalent.
\begin{enumerate}
\item $\tau^-(C)=0$.
\item $K_C\in\Wcal$
\item $C\in\Tcal$.
\end{enumerate}
\end{lem}

\bigskip

 By the same argument as in \cite{AN}, we can also show the following. Since we do not use it in this article, we only introduce the results.
\begin{prop}
The inclusion functor $i_{\Ucal}\colon\underline{\Ucal}=\Ucal/\Wcal\hookrightarrow\underline{\C}$ has a right adjoint $\sigma_{\Ucal}\colon\underline{\C}\rightarrow\underline{\Ucal}$. For any $C\in\C$, $\sigma_{\Ucal}(C)$ is naturally isomorphic to $U_C\in\underline{\Ucal}$ appearing in a distinguished triangle
\[ V_C\rightarrow U_C\overset{u_C}{\longrightarrow}C\overset{v_C}{\longrightarrow}V_C[1]\qquad (U_C\in\Ucal,V_C\in\Vcal). \]

Moreover for any $C\in\C$, the following are equivalent.
\begin{enumerate}
\item $\sigma_{\Ucal}(C)=0$.
\item $U_C\in\Wcal$.
\item $C\in\C^+$.
\end{enumerate}
\end{prop}
Dually, the inclusion functor $i_{\Tcal}\colon\underline{\Tcal}=\Tcal/\Wcal\hookrightarrow\underline{\C}$ admits a left adjoint $\sigma_{\Tcal}\colon\underline{\C}\rightarrow\underline{\Tcal}$. We have a natural isomorphism $\sigma_{\Tcal}(C)\cong T_C\ \ \text{in}\ \underline{\C}$ for any $C\in\C$, and
\[ \sigma_{\Tcal}(C)=0\ \Longleftrightarrow\ T_C\in\Wcal\ \Longleftrightarrow\ C\in\C^-, \]
where $S_C[-1]\rightarrow C\rightarrow T_C\rightarrow S_C$ is a distinguished triangle with $S_C\in\Scal, T_C\in\Tcal$.

\section{$\underline{\Hcal}$ is preabelian}
Now we construct the (co-)kernel of a morphism in $\underline{\Hcal}$.
\begin{dfn}\label{DefMf}
For any $A\in\C^-$, $B\in\C$ and $f\in\C(A,B)$, define $M_f\in\C$ and $m_f\in\C(B,M_f)$ as follows.
\begin{itemize}
\item[1.] Take a distinguished triangle
\[ S_A[-1]\overset{s_A}{\longrightarrow}A\overset{w_A}{\longrightarrow}W_A\rightarrow S_A, \]
\item[2.] then, take a distinguished triangle
\[ S_A[-1]\overset{f\circ s_A}{\longrightarrow}B\overset{m_f}{\longrightarrow}M_f\rightarrow S_A. \]
\end{itemize}
\[
\xy
(0,12)*+{S_A[-1]}="0";
(0,2)*+{A}="2";
(0,-8)*+{W_A}="4";
(14,2)*+{B}="6";
(30,-8)*+{M_f}="8";
(10,8)*+{}="10";
{\ar_{s_A} "0";"2"};
{\ar_{w_A} "2";"4"};
{\ar^{} "0";"6"};
{\ar_{f} "2";"6"};
{\ar^{m_f} "6";"8"};
{\ar@{}|\circlearrowright "2";"10"};
\endxy
\]
\end{dfn}

\begin{prop}\label{PropMf}
For any $A\in\C^-$, $B\in\C$ and $f\in\C(A,B)$, let $B\overset{m_f}{\longrightarrow}M_f$ be as in Definition \ref{DefMf}. Then, we have the following.
\begin{enumerate}
\item $\underline{m}_f\circ\underline{f}=0$.
\item $m_f$ induces a bijection
\[ -\circ\underline{m}_f\colon\underline{\C}(M_f,Y)\overset{\cong}{\longrightarrow}\{ \beta\in\underline{\C}(B,Y)\mid\beta\circ\underline{f}=0 \} \]
for any $Y\in\C^+$.
\item If $B\in\C^-$, then $M_f\in\C^-$.
\end{enumerate}
\end{prop}
\begin{proof}
{\rm (1)} is trivial, since there is a morphism of triangles
\[
\xy
(-22,6)*+{S_A[-1]}="0";
(-7,6)*+{A}="2";
(6,6)*+{W_A}="4";
(18,6)*+{S_A}="6";
(-22,-6)*+{S_A[-1]}="10";
(-7,-6)*+{B}="12";
(6,-6)*+{M_f}="14";
(18,-6)*+{S_A}="16";
(21,-7)*+{.}="17";
{\ar@{=} "0";"10"};
{\ar_{f} "2";"12"};
{\ar_{} "4";"14"};
{\ar@{=} "6";"16"};
{\ar^{s_A} "0";"2"};
{\ar_{} "10";"12"};
{\ar^{w_A} "2";"4"};
{\ar_{m_f} "12";"14"};
{\ar "4";"6"};
{\ar "14";"16"};
{\ar@{}|\circlearrowright "0";"12"};
{\ar@{}|\circlearrowright "2";"14"};
{\ar@{}|\circlearrowright "4";"16"};
\endxy
\]
{\rm (3)} follows from Lemma \ref{LemAB}.

We show {\rm (2)}. Take a distinguished triangle
\[ V_Y\rightarrow W_Y\overset{w_Y}{\longrightarrow}Y\overset{v_Y}{\longrightarrow}V_Y[1]\quad(V_Y\in\Vcal,W_Y\in\Wcal). \]

To show the injectivity, suppose $x\in\C(M_f,Y)$ satisfies $\underline{x}\circ\underline{m}_f=0$.
By definition $x\circ m_f$ factors through some object in $\Wcal$. By (the dual of) Lemma \ref{LemFactor}, $x\circ m_f$ factors through $w_Y$, and thus we obtain a morphism of triangles
\[
\xy
(-22,6)*+{S_A[-1]}="0";
(-6,6)*+{B}="2";
(6,6)*+{M_f}="4";
(21,6)*+{S_A}="6";
(-22,-6)*+{V_Y}="10";
(-6,-6)*+{W_Y}="12";
(6,-6)*+{Y}="14";
(21,-6)*+{V_Y[1].}="16";
{\ar^{f\circ s_A} "0";"2"};
{\ar^{m_f} "2";"4"};
{\ar^{} "4";"6"};
{\ar^{} "0";"10"};
{\ar^{} "2";"12"};
{\ar^{x} "4";"14"};
{\ar^{} "6";"16"};
{\ar_{} "10";"12"};
{\ar_{w_Y} "12";"14"};
{\ar_{v_Y} "14";"16"};
{\ar@{}|\circlearrowright "0";"12"};
{\ar@{}|\circlearrowright "2";"14"};
{\ar@{}|\circlearrowright "4";"16"};
\endxy
\]
Since $\Ext^1(S_A,V_Y)=0$, $x$ factors through $W_Y$, which means $\underline{x}=0$.

To show the surjectivity, suppose $y\in\C(B,Y)$ satisfies $\underline{y}\circ\underline{f}=0$. 
By the same argument as above, we see that $y\circ f$ factors $W_Y$. This implies $y\circ f\circ s_A=0$, since $\Ext^1(S_A,W_Y)=0$. Thus $y$ factors $m_f$.
\[
\xy
(-26,0)*+{S_A[-1]}="0";
(-16,-10)*+{}="1";
(-9,0)*+{B}="2";
(8,0)*+{M_f}="4";
(20,0)*+{S_A}="6";
(0,-12)*+{Y}="-3";
(-1,4)*+{}="3";
{\ar^{f\circ s_A} "0";"2"};
{\ar^{m_f} "2";"4"};
{\ar@/_0.80pc/_{0} "0";"-3"};
{\ar_{} "4";"6"};
{\ar_{y} "2";"-3"};
{\ar^{} "4";"-3"};
{\ar@{}|\circlearrowright"-3";"3"};
{\ar@{}|\circlearrowright"1";"2"};
\endxy
\]
\end{proof}

Dually, we have the following:
\begin{rem}\label{RemLC}
For any $A\in\C$, $B\in\C^+$ and any $f\in\C(A,B)$, take a diagram
\[
\xy
(0,-12)*+{V_B[1]}="0";
(0,-2)*+{B}="2";
(0,8)*+{W_B}="4";
(-16,-2)*+{A}="6";
(-32,8)*+{L_f}="8";
(-10,-8)*+{}="10";
{\ar^{v_B} "2";"0"};
{\ar^{w_B} "4";"2"};
{\ar_{v_B\circ f} "6";"0"};
{\ar^{f} "6";"2"};
{\ar^{\ell_f} "8";"6"};
{\ar@{}|\circlearrowright "2";"10"};
\endxy
\]
where
\begin{eqnarray*}
V_B\rightarrow W_B\overset{w_B}{\longrightarrow} B\overset{v_B}{\longrightarrow} V_B[1]\\
V_B\rightarrow L_f\overset{\ell_f}{\longrightarrow} A\overset{v_B\circ f}{\longrightarrow}V_B[1]
\end{eqnarray*}
are distinguished triangles satisfying $W_B\in\Wcal, V_B\in\Vcal$.

Then, the following holds.
\begin{enumerate}
\item $\underline{f}\circ\underline{\ell}_f=0$.
\item $\ell_f$ induces a bijection
\[ \underline{\ell}_f\circ-\colon\underline{\C}(X,L_f)\overset{\cong}{\longrightarrow}\{ \alpha\in\underline{\C}(X,A)\mid\underline{f}\circ\alpha=0 \} \]
for any $X\in\C^-$.
\item If $A\in\C^+$, then $L_f\in\C^+$.
\end{enumerate}
\end{rem}

\begin{cor}\label{CorPreabel}
For any twin cotorsion pair, $\underline{\Hcal}$ is preabelian. 
\end{cor}
\begin{proof}
First we construct a cokernel.
For any $A,B\in\mathcal{H}$ and any $f\in\C(A,B)$, let $m_f\colon B\rightarrow M_f$ be as in Definition \ref{DefMf}.
Since $A,B\in\C^-$, it follows
\[ \underline{m}_f\circ\underline{f}=0,\quad M_f\in\C^- \]
by Proposition \ref{PropMf}.
By Proposition \ref{PropZC}, there exists $z_{M_f}\colon M_f\rightarrow Z_{M_f}$ which gives a reflection $\underline{z}_{M_f}\colon M_f\rightarrow Z_{M_f}$ of $M_f$ along $\underline{\C}^+\hookrightarrow\underline{\C}$.
By Claim \ref{ClaimZC}, $Z_{M_f}$ satisfies $Z_{M_f}\in\mathcal{H}$.

Then $\underline{z}_{M_f}\circ\underline{m}_f\colon B\rightarrow Z_{M_f}$ gives a cokernel of $\underline{f}$. In fact for any $H\in\Hcal$, there is a bijection
\[ -\circ\underline{z}_{M_f}\circ\underline{m}_f\ \colon\ \underline{\C}(Z_{M_f},H)\overset{\cong}{\longrightarrow}\underline{\C}(M_f,H)\overset{\cong}{\longrightarrow}\{ \beta\in\underline{\C}(B,H)\mid \beta\circ\underline{f}=0 \}. \]
%\begin{eqnarray*}
%\underline{\C}(Z_{M_f},H)&\overset{\cong}{\longrightarrow}&\underline{\C}(M_f,H)\\
%&\overset{\cong}{\longrightarrow}&\{ \beta\in\underline{\C}\mid \beta\circ\underline{f}=0 \}.
%\end{eqnarray*}
\[
\xy
(-18,0)*+{A}="2";
(0,0)*+{B}="6";
(-2,10)*+{}="7";
(14,10)*+{H}="8";
(9,-7)*+{M_f}="4";
(18,-14)*+{Z_{M_f}}="10";
{\ar "6";"4"};
{\ar_{\underline{f}} "2";"6"};
{\ar_{\underline{m}_f} "6";"4"};
{\ar_{\underline{z}_{M_f}} "4";"10"};
{\ar_{\beta} "6";"8"};
{\ar@{}|{\circlearrowright} "6";"7"};
{\ar@/^0.80pc/^{0} "2";"8"};
\endxy
\]

A kernel of $\underline{f}\in\underline{\Hcal}(A,B)$ is constructed dually. Let $L\overset{\ell_f}{\longrightarrow}A$ be as in Remark \ref{RemLC}, and let $K_{L_f}\overset{k_{L_f}}{\longrightarrow}L_f$ be as in Definition \ref{DefKC}. Then $\underline{\ell}_f\circ\underline{k}_{L_f}$ gives a kernel of $\underline{f}$.
\end{proof}

\begin{cor}\label{CorEpi1}
Let $f\in\Hcal(A,B)$ be a morphism in $\Hcal$. The following are equivalent.
\begin{enumerate}
\item $\underline{f}\in\underline{\Hcal}(A,B)$ is epimorphic.
\item $Z_{M_f}\in\Wcal$.
\item $M_f\in\Ucal$.
\end{enumerate}
\end{cor}
\begin{proof}
{\rm (1)} is equivalent to {\rm (2)}, since $\mathrm{cok} \underline{f}\cong Z_{M_f}$ in $\underline{\Hcal}$. Also {\rm (2)} is equivalent to {\rm (3)} by Lemma \ref{LemUZ}.
\end{proof}

\begin{cor}\label{CorEpi2}
Let $f\in\Hcal(A,B)$ be a morphism in $\Hcal$. If a distinguished triangle
\[ A\overset{f}{\longrightarrow}B\overset{g}{\longrightarrow}C\rightarrow A[1] \]
admits a factorization
\[
\xy
(-8,4)*+{B}="0";
(8,4)*+{C}="2";
(0,-8)*+{U}="4";
(0,6)*+{}="3";
{\ar^{g} "0";"2"};
{\ar_{b} "0";"4"};
{\ar_{c} "4";"2"};
{\ar@{}|\circlearrowright "3";"4"};
\endxy
\]
for some $U\in\Ucal$, then $\underline{f}\in\underline{\Hcal}(A,B)$ is epimorphic.
\end{cor}
\begin{proof}
By the definition of $B\overset{m_f}{\longrightarrow}M_f$, there is a commutative diagram made of distinguished triangles as follows (Definition \ref{DefMf}).
\[
\xy
(-16,16)*+{S_A[-1]}="0";
(-15,-1)*+{A}="2";
(-13.5,-17.5)*+{W_A}="4";
(-4,2)*+{B}="6";
(2.3,-5.3)*+{M_f}="8";
(16.5,8.5)*+{C}="10";
(-8.5,-7.5)*+_{_{\circlearrowright}}="12";
(-11.5,5.5)*+_{_{\circlearrowright}}="14";
(4.3,0.5)*+_{_{\circlearrowright}}="14";
{\ar_{s_A} "0";"2"};
{\ar_{w_A} "2";"4"};
{\ar^{} "0";"6"};
{\ar_{f} "2";"6"};
{\ar_{m_f} "6";"8"};
{\ar^{g} "6";"10"};
{\ar_{e} "4";"8"};
{\ar_{d} "8";"10"};
\endxy
\]
By Corollary \ref{CorEpi1}, it suffices to show $M_f\in\Ucal$.
Take any $V^{\dag}\in\Vcal$ and $v\in\C(M_f,V^{\dag}[1])$. Since $v\circ e=0$ by $\Ext^1(W_A,V^{\dag})=0$, there exists $v^{\prime}\in\C(C,V^{\dag}[1])$ satisfying $v^{\prime}\circ d=v$.
\[
\xy
(-24,0)*+{W_A}="0";
(-9,0)*+{M_f}="2";
(8,0)*+{C}="4";
(0,-12)*+{V^{\dag}[1]}="-3";
(-1,-2)*+{}="5";
(0,12)*+{B}="6";
(-1,4)*+{}="3";
{\ar^{e} "0";"2"};
{\ar^{d} "2";"4"};
{\ar_{v} "2";"-3"};
{\ar^{v^{\prime}} "4";"-3"};
{\ar_{m_f} "6";"2"};
{\ar^{g} "6";"4"};
{\ar@{}|\circlearrowright"5";"6"};
{\ar@{}|\circlearrowright"-3";"3"};
\endxy
\]
Since $g$ factors through $U\in\Ucal$, it follows $v\circ m_f=v^{\prime}\circ d\circ m_f=v^{\prime}\circ g=0$.
Thus $v$ factors through $S_A$,
\[
\xy
(-42,0)*+{S_A[-1]}="-2";
(-24,0)*+{B}="0";
(-16,-10)*+{}="1";
(-9,0)*+{M_f}="2";
(8,0)*+{S_A}="4";
%(20,0)*+{S_A}="6";
(0,-12)*+{V^{\dag}[1]}="-3";
(-1,4)*+{}="3";
{\ar^{f\circ s_A} "-2";"0"};
{\ar^{m_f} "0";"2"};
{\ar@/_0.80pc/_{0} "0";"-3"};
{\ar_{} "2";"4"};
{\ar_{v} "2";"-3"};
{\ar^{} "4";"-3"};
{\ar@{}|\circlearrowright"-3";"3"};
{\ar@{}|\circlearrowright"1";"2"};
\endxy
\]
which means $v=0$, since $\Ext^1(S_A,V^{\dag})=0$.
\end{proof}
\begin{rem}
Duals of Corollary \ref{CorEpi1} and \ref{CorEpi2} also hold for monomorphisms in $\underline{\Hcal}$.
\end{rem}

\section{$\underline{\Hcal}$ is semi-abelian}

%First, we give a characterization of a cokernel morphism in $\underline{\Hcal}$.
\begin{lem}\label{LemCokMorph}
Let $\beta\in\underline{\Hcal}(B,C)$ be any morphism.
%The following are equivalent.
%\begin{enumerate}
%\item 
If $\beta$ is a cokernel morphism, namely, if there exists a morphism $f\in\Hcal(A,B)$ such that $\beta=\mathrm{cok}\underline{f}$, 
%\item 
then there exist $g\in\Hcal(B,C^{\prime})$ and an isomorphism $\eta\in\underline{\Hcal}(C,C^{\prime})$ such that
\begin{itemize}
\item[{\rm (i)}] $\eta$ is compatible with $\beta$ and $\underline{g}$,
\[
\xy
(-8,6)*+{B}="0";
(8,6)*+{C}="2";
(0,-6)*+{C^{\prime}}="4";
(0,10)*+{}="3";
{\ar^{\beta} "0";"2"};
{\ar_{\underline{g}} "0";"4"};
{\ar^{\cong}_{\eta} "2";"4"};
{\ar@{}|{\circlearrowright} "3";"4"};
\endxy
\]
\item[{\rm (ii)}] $g$ admits a distinguished triangle
\[ S[-1]\overset{s}{\longrightarrow}B\overset{g}{\longrightarrow}C^{\prime}\rightarrow S \]
with $S\in\Scal$.
\end{itemize}
%\end{enumerate}
\end{lem}
\begin{proof}
%Suppose {\rm (1)} holds, and 
Take a morphism $f\in\Hcal(A,B)$ such that $\beta=\mathrm{cok}\underline{f}$. As shown in Corollary \ref{CorPreabel}, $\mathrm{cok}\underline{f}$ is given by $\underline{z}_{M_f}\circ\underline{m}_f$.
\[
\xy
(0,14)*+{S_A[-1]}="0";
(0,0)*+{A}="2";
(0,-14)*+{W_A}="4";
(14,0)*+{B}="6";
(30,-14)*+{M_f}="8";
(10,8)*+{}="10";
{\ar_{s_A} "0";"2"};
{\ar_{v_A} "2";"4"};
{\ar^{} "0";"6"};
{\ar_{f} "2";"6"};
{\ar^{m_f} "6";"8"};
{\ar@{}|\circlearrowright "2";"10"};
\endxy
\qquad
\xy
(-16,16)*+{S[-1]}="0";
(-14.8,-1)*+{U}="2";
(-13.5,-17.5)*+{T}="4";
(-4,2)*+{M_f}="6";
(2.2,-5.2)*+{Z_{M_f}}="8";
(16.5,8.5)*+{V[1]}="10";
(-8.5,-7.5)*+_{_{\circlearrowright}}="12";
(-11.5,5.5)*+_{_{\circlearrowright}}="14";
(4.3,0.5)*+_{_{\circlearrowright}}="14";
{\ar_{} "0";"2"};
{\ar_{} "2";"4"};
{\ar^{} "0";"6"};
{\ar_{} "2";"6"};
{\ar_{z_{M_f}} "6";"8"};
{\ar^{} "6";"10"};
{\ar_{} "4";"8"};
{\ar_{} "8";"10"};
\endxy
\]
Thus there exists an isomorphism $\eta\in\underline{\Hcal}(C,Z_{M_f})$ compatible with $\underline{z}_{M_f}\circ\underline{m}_f$ and $\beta$.
It suffices to show $g=z_{M_f}\circ m_f$ satisfies condition {\rm (ii)}. % in {\rm (2)}.
If we complete $g$ into a distinguished triangle
\[ Q[-1]\rightarrow B\overset{z_{M_f}\circ m_f}{\longrightarrow}Z_{M_f}\rightarrow Q, \]
then by the octahedral axiom, we have a distinguished triangle
\[ S_A\rightarrow Q\rightarrow S\rightarrow S_A[1], \]
which implies $Q\in\Scal$.
\[
\xy
(-16,16)*+{B}="0";
(-14.6,-1)*+{M_f}="2";
(-13.5,-17.5)*+{S_A}="4";
(-4,2)*+{Z_{M_f}}="6";
(2,-5.1)*+{Q}="8";
(16.7,8.5)*+{S}="10";
(-8.5,-7.5)*+_{_{\circlearrowright}}="12";
(-11.5,5.5)*+_{_{\circlearrowright}}="14";
(4.3,0.5)*+_{_{\circlearrowright}}="14";
{\ar_{m_f} "0";"2"};
{\ar_{} "2";"4"};
{\ar^{} "0";"6"};
{\ar_{z_{M_f}} "2";"6"};
{\ar_{} "6";"8"};
{\ar^{} "6";"10"};
{\ar_{} "4";"8"};
{\ar_{} "8";"10"};
\endxy
\]

%Conversely, suppose {\rm (2)} is satisfied. We may assume $\eta=\mathrm{id}_C$ and $\beta=\underline{g}$. Then by the definition of $m_s\colon B\rightarrow M_s$, it can be taken as $g\colon B\rightarrow C$.
%\[
%\xy
%(0,12)*+{S[-1]}="0";
%(0,2)*+{S[-1]}="2";
%(0,-8)*+{0}="4";
%(14,2)*+{B}="6";
%(30,-8)*+{M_s}="8";
%(10,8)*+{}="10";
%{\ar_{\cong} "0";"2"};
%{\ar_{} "2";"4"};
%{\ar^{} "0";"6"};
%{\ar_{s} "2";"6"};
%{\ar^{m_s} "6";"8"};
%{\ar@{}|\circlearrowright "2";"10"};
%\endxy
%\]
%Since $M_s\cong C\in\Hcal$, we see that $\underline{z}_{M_s}\colon M_s\rightarrow Z_{M_s}$ is an isomorphism. Thus $\beta=\underline{g}$ agrees with $\mathrm{cok}\,\underline{s}$.
\end{proof}

%%%%%%%%%%%%%%%%%%%%%%%%%%%%%%%%%%%%%%%%%%%%%%%%%%%%%%%
%%%%%%%%%%%%%%%%%%%%%%%%%%%%%%%%%%%%%%%%%%%%%%%%%%%%%%%
%%%single cotors'̏ꍇ'Éabel'ƂȂ邱'Æ'ðŒn'ŏo'µ'½'¢%%%
%%%%%%%%%%%%%%%%%%%%%%%%%%%%%%%%%%%%%%%%%%%%%%%%%%%%%%%
%%%%%%%%%%%%%%%%%%%%%%%%%%%%%%%%%%%%%%%%%%%%%%%%%%%%%%%

\begin{lem}\label{Lem-Epim}
Suppose $X\in\C^-,B\in\Hcal$ and $x\in\C(X,B)$ admit a distinguished triangle
\[ X\overset{x}{\longrightarrow}B\rightarrow U\rightarrow X[1] \]
with some $U\in\Ucal$. Then, the unique morphism $\zeta\in\underline{\Hcal}(Z_X,B)$ satisfying $\zeta\circ\underline{z}_X=\underline{x}$ $($Proposition \ref{PropZC}$)$ becomes epimorphic.
\end{lem}
\begin{proof}
As in Proposition \ref{PropZC} (or, dual of the proof of Proposition \ref{PropKC}), we see that there exists $b\in\C(Z_X,B)$ satisfying $\zeta=\underline{b}$ and $b\circ z_X=x$. If we complete $b$ into a distinguished triangle
\[ C[-1]\rightarrow Z_X\overset{b}{\longrightarrow}B\overset{c}{\longrightarrow}C, \]
then $c$ factors through $U$.
\[
\xy
(-12,8)*+{X}="0";
(-12,-8)*+{Z_X}="2";
(-16,0)*+{}="3";
(0,0)*+{B}="4";
(16,0)*+{}="5";
(12,8)*+{C}="6";
(12,-8)*+{U}="8";
{\ar_{z_X} "0";"2"};
{\ar^{x} "0";"4"};
{\ar_{b} "2";"4"};
{\ar^{c} "4";"6"};
{\ar^{} "4";"8"};
{\ar_{} "8";"6"};
{\ar@{}|\circlearrowright "3";"4"};
{\ar@{}|\circlearrowright "4";"5"};
\endxy
\]
Thus Lemma \ref{Lem-Epim} follows from Corollary \ref{CorEpi2}.
\end{proof}

\begin{lem}\label{LemPullBack}
Let
\begin{equation}\label{PullBackPre}
\xy
(-6,6)*+{A}="0";
(6,6)*+{B}="2";
(-6,-6)*+{C}="4";
(6,-6)*+{D}="6";
(0,0)*+{\square}="5";
{\ar^{\alpha} "0";"2"};
{\ar_{\beta} "0";"4"};
{\ar^{\gamma} "2";"6"};
{\ar_{\delta} "4";"6"};
\endxy
\end{equation}
be a pullback diagram in $\underline{\Hcal}$. If there exist $X\in\C^-, x_B\in\C(X,B), x_C\in\C(X,C)$ which satisfies the following conditions, then $\alpha$ is epimorphic.
\begin{itemize}
\item[{\rm (i)}] The following diagram is commutative.
\[
\xy
(-6,6)*+{X}="0";
(6,6)*+{B}="2";
(-6,-6)*+{C}="4";
(6,-6)*+{D}="6";
{\ar^{\underline{x}_B} "0";"2"};
{\ar_{\underline{x}_C} "0";"4"};
{\ar^{\gamma} "2";"6"};
{\ar_{\delta} "4";"6"};
{\ar@{}|\circlearrowright "0";"6"};
\endxy
\]
\item[{\rm (ii)}] There exists a distinguished triangle $X\overset{x_B}{\longrightarrow}B\rightarrow U\rightarrow X[1]$ with $U\in\Ucal$.
\end{itemize}
\end{lem}
\begin{proof}
Take $X\overset{z_X}{\longrightarrow}Z_X$ as in Definition \ref{DefZC}. By the adjointness, there exist $\zeta_B\in\underline{\Hcal}(Z_X,B)$ and $\zeta_C\in\underline{\Hcal}(Z_X,C)$ satisfying
\[ \zeta_B\circ\underline{z}_X=\underline{x}_B,\ \ \zeta_C\circ\underline{z}_X=\underline{x}_C. \] By Lemma \ref{Lem-Epim}, $\zeta_B$ is epimorphic.
From $\gamma\circ\underline{x}_B=\delta\circ\underline{x}_C$, it follows $\gamma\circ\zeta_B=\delta\circ\zeta_C$.
\[
\xy
(-12,12)*+{X}="0";
(10,6)*+{B}="2";
(-6,-10)*+{C}="4";
(-4,4)*+{}="5";
(10,-10)*+{D}="6";
(-1,14)*+{}="7";
(-2,2)*+{Z_X}="8";
(-14,1)*+{}="9";
{\ar@/^0.40pc/^{\underline{x}_B} "0";"2"};
{\ar@/_0.40pc/_{\underline{x}_C} "0";"4"};
{\ar^{\underline{z}_X} "0";"8"};
{\ar@/^0.20pc/_{\zeta_B} "8";"2"};
{\ar@/_0.20pc/^{\zeta_C} "8";"4"};
{\ar^{\gamma} "2";"6"};
{\ar_{\delta} "4";"6"};
{\ar@{}|\circlearrowright "5";"6"};
{\ar@{}|\circlearrowright "7";"8"};
{\ar@{}|\circlearrowright "8";"9"};
\endxy \]
Since $(\ref{PullBackPre})$ is a pullback diagram in $\underline{\Hcal}$, there exists $\zeta\in\underline{\Hcal}(Z_X,A)$ which satisfies $\alpha\circ\zeta=\zeta_B$ and $\beta\circ\zeta=\zeta_C$. 
\[
\xy
(-14,14)*+{Z_X}="-2";
(-2,14)*+{}="-1";
(-6,6)*+{A}="0";
(-14,2)*+{}="1";
(6,6)*+{B}="2";
(-6,-6)*+{C}="4";
(6,-6)*+{D}="6";
(0,0)*+{\square}="5";
{\ar_{\alpha} "0";"2"};
{\ar^{\beta} "0";"4"};
{\ar^{\gamma} "2";"6"};
{\ar_{\delta} "4";"6"};
{\ar_{\zeta} "-2";"0"};
{\ar@/^0.60pc/^{\zeta_B} "-2";"2"};
{\ar@/_0.60pc/_{\zeta_C} "-2";"4"};
{\ar@{}|\circlearrowright "0";"1"};
{\ar@{}|\circlearrowright "0";"-1"};
\endxy
\]
Since $\zeta_B$ is epimorphic, $\alpha$ is also an epimorphism.
\end{proof}

\begin{thm}\label{ThmSemiAbel}
For any twin cotorsion pair, $\underline{\Hcal}$ is semi-abelian.
\end{thm}
\begin{proof}
By duality, we only show $\underline{\Hcal}$ is {\it left} semi-abelian. Assume we are given a pullback diagram
\begin{equation}
\label{PullBackLast}
\xy
(-6,6)*+{A}="0";
(6,6)*+{B}="2";
(-6,-6)*+{C}="4";
(6,-6)*+{D}="6";
(0,0)*+{\square}="5";
{\ar^{\alpha} "0";"2"};
{\ar_{\beta} "0";"4"};
{\ar^{\gamma} "2";"6"};
{\ar_{\delta} "4";"6"};
\endxy
\end{equation}
in $\underline{\Hcal}$, where $\delta$ is a cokernel morphism. It suffices to show $\alpha$ becomes epimorphic.

By Lemma \ref{LemCokMorph}, replacing $D$ by an isomorphic one if necessary, we may assume there exists $d\in\Hcal(C,D)$ satisfying $\delta=\underline{d}$, which admits a distinguished triangle
\[ S[-1]\rightarrow C\overset{d}{\longrightarrow}D\overset{s}{\longrightarrow}S \]
with $S\in\Scal$.
If we take $c\in\Hcal(B,D)$ satisfying $\gamma=\underline{c}$, and complete $s\circ c$ into a distinguished triangle
\[ S[-1]\rightarrow X\overset{x_B}{\longrightarrow}B\overset{s\circ c}{\longrightarrow}S, \]
then $c\circ x_B$ factors through $d$. In fact, there exists $x_C\in\C(X,C)$ which gives a morphism of triangles as follows.
\[
\xy
(-22,6)*+{S[-1]}="0";
(-6,6)*+{X}="2";
(6,6)*+{B}="4";
(18,6)*+{S}="6";
(-22,-6)*+{S[-1]}="10";
(-6,-6)*+{C}="12";
(6,-6)*+{D}="14";
(18,-6)*+{S}="16";
{\ar^{} "0";"2"};
{\ar^{x_B} "2";"4"};
{\ar^{s\circ c} "4";"6"};
{\ar@{=} "0";"10"};
{\ar_{x_C} "2";"12"};
{\ar^{c} "4";"14"};
{\ar@{=} "6";"16"};
{\ar_{} "10";"12"};
{\ar_{d} "12";"14"};
{\ar_{s} "14";"16"};
{\ar@{}|\circlearrowright "0";"12"};
{\ar@{}|\circlearrowright "2";"14"};
{\ar@{}|\circlearrowright "4";"16"};
\endxy
\]
By Lemma \ref{LemAB}, we have $X\in\C^-$. Thus $\alpha$ becomes epimorphic by Lemma \ref{LemPullBack}.
\end{proof}

\section{The case where $\underline{\Hcal}$ becomes integral}

In the rest, additionally we assume that $(\Scal,\Tcal),(\Ucal,\Vcal)$ satisfies
\begin{equation}\label{Double}
\Ucal\subseteq\Scal\ast\Tcal\quad \text{or}\quad \Tcal\subseteq\Ucal\ast\Vcal.
\end{equation}
This condition is satisfied, for example in the following cases.
\begin{ex}\label{ExDouble}
A twin cotorsion pair $(\Scal,\Tcal),(\Ucal,\Vcal)$ satisfies $(\ref{Double})$ in the following cases.
\begin{enumerate}
\item $\Ucal=\Scal$. Namely, $(\Scal,\Tcal)=(\Ucal,\Vcal)$ is a single cotorsion pair.
\item $\Ucal=\Tcal$. For example, Buan and Marsh's triplet $(\add (T)[1],\Xcal_T,(\Xcal_T)^{\perp}[-1])$.
\item $(\Scal,\Tcal)$ is a co-$t$-structure. In this case, $\Scal\ast\Tcal=\C$.
\item $(\Ucal,\Vcal)$ is a co-$t$-structure. In this case, $\Ucal\ast\Vcal=\C$.
\end{enumerate}
\end{ex}

Remark that the following holds.
\begin{fact}\label{LRintegral}$($\cite{R}$)$ A semi-abelian category $\A$ is left integral if and only if $\A$ is right integral.
\end{fact}

\begin{thm}\label{ThmIntegral}
If a twin cotorsion pair $\Twin$ satisfies $(\ref{Double})$, then $\underline{\Hcal}$ becomes integral.
\end{thm}
\begin{proof}
By duality, it suffices to show that $\Ucal\subseteq\Scal\ast\Tcal$ implies left integrality. 

Let $b\in\Hcal(B,D)$ and $c\in\Hcal(C,D)$ be morphisms satisfying $\beta=\underline{b}$ and $\delta=\underline{d}$.
Since $\delta$ is epimorphic, if we take $D\overset{m_d}{\longrightarrow}M_d$ as in Definition \ref{DefMf}, then $M_d\in\Ucal$ by Corollary \ref{CorEpi1}.
\[
\xy
(0,12)*+{S_C[-1]}="0";
(0,2)*+{C}="2";
(0,-8)*+{W_C}="4";
(14,2)*+{D}="6";
(30,-8)*+{M_d}="8";
(10,8)*+{}="10";
{\ar_{s_C} "0";"2"};
{\ar_{w_C} "2";"4"};
{\ar^{} "0";"6"};
{\ar_{d} "2";"6"};
{\ar^{m_d} "6";"8"};
{\ar@{}|\circlearrowright "2";"10"};
\endxy
\]
By assumption $\Ucal\subseteq\Scal\ast\Tcal$, there exists a distinguished triangle
\[ S_0\overset{s_0}{\longrightarrow}M_d\overset{t_0}{\longrightarrow}T_0\rightarrow S_0[1] \]
with $S_0\in\Scal, T_0\in\Tcal$.

If we take a distinguished triangle
\[ S_B[-1]\overset{s_B}{\longrightarrow}B\overset{w_B}{\longrightarrow}W_B\rightarrow S_B\quad(S_B\in\Scal,W_B\in\Wcal), \]
then by $\Ext^1(S_B,T_0)=0$, $m_d\circ c\circ s_B$ factors through $s_0$. Namely, there exists $g\in\C(S_B[-1],S_0)$ which makes the following diagram commutative.
\[
\xy
(-8,9)*+{S_B[-1]}="0";
(8,9)*+{S_0}="2";
(-8,0)*+{B}="4";
(-8,-9)*+{D}="6";
(8,-9)*+{M_d}="8";
{\ar^{g} "0";"2"};
{\ar_{s_B} "0";"4"};
{\ar_{c} "4";"6"};
{\ar_{m_d} "6";"8"};
{\ar^{s_0} "2";"8"};
{\ar@{}|\circlearrowright "0";"8"};
\endxy
\]
If we complete $g$ into a distinguished triangle
\[ S_0[-1]\rightarrow X\overset{f_B}{\longrightarrow}S_B[-1]\overset{g}{\longrightarrow}S_0, \]
then $X\in\Scal[-1]\subseteq\C^-$. Moreover there exists $f_C\in\C(X,S_C[-1])$ satisfying $d\circ s_C\circ f_C=c\circ s_B\circ f_B$.
\[
\xy
(-22,6)*+{S_0[-1]}="0";
(-6,6)*+{X}="2";
(10,6)*+{S_B[-1]}="4";
(24,6)*+{S_0}="6";
(-22,-6)*+{M_d[-1]}="10";
(-6,-6)*+{S_C[-1]}="12";
(10,-6)*+{D}="14";
(24,-6)*+{M_d}="16";
{\ar^{} "0";"2"};
{\ar^{f_B} "2";"4"};
{\ar^{g} "4";"6"};
{\ar "0";"10"};
{\ar_{f_C} "2";"12"};
{\ar|*+{_{c\circ s_B}} "4";"14"};
{\ar^{s_B} "6";"16"};
{\ar_{} "10";"12"};
{\ar_{d\circ s_C} "12";"14"};
{\ar_{m_d} "14";"16"};
{\ar@{}|\circlearrowright "0";"12"};
{\ar@{}|\circlearrowright "2";"14"};
{\ar@{}|\circlearrowright "4";"16"};
\endxy
\]
Thus we have a commutative diagram
\[
\xy
(-8,7)*+{X}="0";
(8,7)*+{B}="2";
(-8,-7)*+{C}="4";
(8,-7)*+{D}="6";
(11,-8)*+{.}="7";
{\ar^{s_B\circ f_B} "0";"2"};
{\ar_{s_C\circ f_C} "0";"4"};
{\ar^{c} "2";"6"};
{\ar_{d} "4";"6"};
{\ar@{}|\circlearrowright "0";"6"};
\endxy
\]
If we complete $s_B\circ f_B$ into a distinguished triangle
\[ X\overset{s_B\circ f_B}{\longrightarrow}B\rightarrow Q\rightarrow X[1], \]
then by the octahedral axiom, we have $Q\in\Ucal$. Thus by Lemma \ref{LemPullBack}, $\alpha$ becomes epimorphic.
\[
\xy
(-22,18)*+{X}="0";
(-18,-2)*+{S_B[-1]}="2";
(-14,-18)*+{S_0}="4";
(-4,2)*+{B}="6";
(2.4,-5)*+{Q}="8";
(16.5,8.5)*+{W_B}="10";
(-8.5,-7.5)*+_{_{\circlearrowright}}="12";
(-14,5.5)*+_{_{\circlearrowright}}="14";
(4.3,0.5)*+_{_{\circlearrowright}}="14";
{\ar_{f_B} "0";"2"};
{\ar_{} "2";"4"};
{\ar^{} "0";"6"};
{\ar_{s_B} "2";"6"};
{\ar_{w_B} "6";"8"};
{\ar^{} "6";"10"};
{\ar_{} "4";"8"};
{\ar_{} "8";"10"};
\endxy
\]
\end{proof}

\end{document}